\documentclass[12pt,leqno]{amsart}
\usepackage{amsmath}
\usepackage{amssymb}
\newtheorem{definition}[equation]{Definition}

\newtheorem{lemma}[equation]{Lemma}

\newtheorem{proposition}[equation]{Proposition}

\newtheorem{theorem}[equation]{Theorem}

\renewcommand{\dim}{{\mathrm{dim}}}
\newcommand{\rank}{\mathrm{rank}}
\newcommand{\C}{{\mathbb{C}}}

\renewcommand{\P}{{\mathbb{P}}}

\newcommand{\Z}{{\mathbb{Z}}}
\newcommand{\F}{{\mathbb{F}}}
\newcommand{\B}{{\mathbb{B}}}
\newcommand{\bbf}{{\mathbf{f}}}
\begin{document}
\title[Truncated second main theorem for non-Archimedean meromorphic maps]{Truncated second main theorem for non-Archimedean meromorphic maps}

\author[S. D. Quang]{Si Duc Quang}
\address{Department of Mathematics\\ Hanoi National University of Education\\ 136-Xuan Thuy, Cau Giay, Hanoi, Vietnam}
\email{quangsd@hnue.edu.vn}

\date{}

\begin{abstract}
Let $\F$ be an algebraically closed field of characteristic $p\ge 0$, which is complete with respect to a non-Archimedean absolute value. Let $V$ be a projective subvariety of $\P^M(\F)$. In this paper, we will prove some second main theorems for non-Archimedean meromorphic maps of $\F^m$ into $V$ intersecting a family of hypersurfaces in $N-$subgeneral position with truncated counting functions. 
\end{abstract}



\maketitle      
\section{Introduction and Main results} 

Let $\F$ be an algebraically closed field of characteristic $p\ge 0$, which is complete with respect to a non-Archimedean absolute value. Let $N\geq n$ and $q\geq N+1.$ Let $H_1,\ldots, H_q$ be hyperplanes in $\P^n(\mathbb F).$
The family of hyperplanes $\{H_1\}_{i=1}^q$ is said to be in $N$-subgeneral position in $\P^n(\mathbb F)$ if 
$H_{j_0}\cap\cdots\cap H_{j_N}=\varnothing$ for every $1\leq j_0<\cdots<j_N\leq q.$ 

In 2017, Yan \cite{Yan} proved a truncated second main theorem for a non-Archimedean meromorphic map into $\P^n(\mathbb F)$ with a family of hyperplanes in subgeneral position. With the standart notations on the Nevanlinna theory for non-Archimedean meromorphic maps, his result is stated as follows.

\vskip0.2cm
\noindent
{\bf Theorem A} (cf. \cite[Theorem 4.6]{Yan}) {\it Let $\F$ be an algebraically closed field of characteristic $p\ge 0$, which is complete with respect to a non-Archimedean absolute value. Let $f:\F^m\rightarrow\P^n(\F)$ be a linearly non-degenerate non-Archimedean meromorphic map with index of independence $s$ and $\mathrm{rank}f = k$. Let $H_1,\ldots,H_q$ be hyperplanes in $\P^n(\F)$ in $N$-subgeneral position $(N\ge n)$. Then, for all $r\ge 1$,
$$(q-2N+n-1)T_f(r)\le \sum_{i=1}^{q}N^{(a)}_{f}(H_i,r)-\frac{N+1}{n+1}\log r+O(1),$$
where 
$$a=\begin{cases}
p^{s-1}(n-k+1)&\text{ if }p>0,\\
n-k+1&\text{ if }p=0.
\end{cases}$$}

Here, the index of independence $s$ and the $\mathrm{rank}f$ are defined in Section 2 (Definition \ref{def3}). 

Also, in 2017, An and Quang \cite{AQ} proved a truncated second main theorem for meromorphic mappings from $\C^m$ into a projective variety $V\subset\P^M(\C)$ with hypersurfaces. Motivated by the methods of Yan \cite{Yan} and An-Quang \cite{AQ}, our aim in this article is to generalize Theorem A to the case where the map $f$ is from $\F^m$ into an arbitrary projective variety $V$ of dimension $n$ in $\P^M(\F)$ and the hyperplanes are replaced by hypersurfaces of $\P^M(\F)$ in $N$-subgeneral position with respect to $V$. 

Firstly, we give the following definitions.

\vskip0.2cm
\noindent
{\bf Definition B.} {\it Let $V$ be a projective subvariety of $\P^M(\mathbb F)$ of dimension $n\ (n\le M)$. Let $Q_1,\ldots,Q_q\ (q\ge n+1)$ be $q$ hypersurfaces in $\P^M(\mathbb F)$. The family of hypersurfaces $\{Q_i\}_{i=1}^q$ is said to be in $N$-subgeneral position with respect to $V$ if
$$ V\cap (\bigcap_{j=1}^{N+1}Q_{i_j})=\varnothing\ \text{ for any }1\le i_1<\cdots <i_{N+1}\le q.$$
If $N=n$, we just say $\{Q_i\}_{i=1}^q$ is in \textit{general position} with respect to $V.$}

Now, let $V$ be as above and let $d$ be a positive integer. We denote by $I(V)$ the ideal of homogeneous polynomials in $\mathbb F[x_0,\ldots,x_M]$ defining $V$ and by $H_d$ the $\mathbb F$-vector space of all homogeneous polynomials in $\mathbb F[x_0,\ldots,x_M]$ of degree $d.$  Define 
$$I_d(V):=\dfrac{H_d}{I(V)\cap H_d}\text{ and }H_V(d):=\dim_{\F} I_d(V).$$
Then $H_V(d)$ is called the Hilbert function of $V$. Each element of $I_d(V)$ which is an equivalent class of an element $Q\in H_d,$ will be denoted by $[Q]$, 

\vskip0.2cm
\noindent
{\bf Definition C.} {\it Let $f:\F^m\rightarrow V$ be a non-Archimedean meromorphic map with a reduced representation $\bbf =(f_0,\ldots,f_M)$. We say that $f$ is degenerate over $I_d(V)$ if there is $[Q]\in I_d(V)\setminus \{0\}$ such that $Q(\bbf)\equiv 0.$ Otherwise, we say that $f$ is non-degenerate over $I_d(V)$.}

We will generalize Theorem A to the following.
\begin{theorem}\label{1.1} 
Let $V$ be a projective subvariety of $\P^M(\F)$ of dimension $n\ (n\le M)$. Let $\{Q_i\}_{i=1}^q$ be hypersurfaces of $\P^M(\F)$ in $N$-subgeneral position with respect to $V$ with $\deg Q_i=d_i\ (1\le i\le q)$. Let $d$ be the least common multiple of $d_i'$s. Let $f$ be a non-Archimedean meromorphic map of $\F^m$ into $V$, which is non-degenerate over $I_d(V)$ with the $d^{th}$-index of non-degeneracy $s$ and $\mathrm{rank}f = k$. Then, for all $r\ge 1$,
$$\left(q-\dfrac{(2N+n-1)H_d(V)}{n+1}\right)T_f(r)\le \sum_{i=1}^{q}\dfrac{1}{d_i}N^{(\kappa_0)}_{f}(Q_i,r)-\frac{N(H_d(V)-1)}{nd}\log r+O(1),$$
where 
$$\kappa_0=\begin{cases}
p^{s-1}(H_d(V)-k)&\text{ if }p>0,\\
H_d(V)-k&\text{ if }p=0.
\end{cases}$$
\end{theorem}
Here, the $d^{th}$-index of non-degeneracy $s$ is defined in Section 2 (Definition \ref{def3}). Note that, in the case where $V=\P^n(\C),d=1,H_d(V)=n+1$, our result will give back Theorem A.

For the case of counting function without truncation level, we will prove the following.
\begin{theorem}\label{1.2} 
Let $V$ be a arbitrary projective subvariety of $\P^M(\F)$. Let $\{Q_i\}_{i=1}^q$ be hypersurfaces of $\P^M(\F)$ in $N$-subgeneral position with respect to $V$. Let $f$ be a non-constant non-Archimedean meromorphic map of $\F^m$ into $V$. Then, for any $r>0$,
$$(q-N)T_f(r)\le \sum_{i=1}^{q}\dfrac{1}{\deg Q_i}N_{f}(Q_i,r)+O(1),$$
where the quantity $O(1)$ depends only on $\{Q_i\}_{i=1}^q$.
\end{theorem}

We see that, the above result is a generalization of the previous results in \cite{A,R01}.

\section{Basic notions and auxiliary results}

In this section, we will recall some basic notions from Nevanlinna theory for non-Archimedean meromorphic maps due to Cherry-Ye \cite{CY} and Yan \cite{Yan}.

\noindent
{\bf 2.1. Non-Archimedean meromorphic function.} \ Let $\F$ be an algebraically closed field of characteristic $p$, complete with respect to a non-Archimedean absolute value $|\ |$.
\ We set $\|z\| = \max_{1\le i\le m}|z_i|$ for
$z = (z_1,\dots,z_m)\in\F^m$ and define
\begin{align*}
\B^m(r) := \{ z \in \F^m ; \|z\| < r\}.
\end{align*}
For a multi-index $\gamma=(\gamma_1,\ldots,\gamma_m)\in\Z^m_{\ge 0}$, define 
$$z^\gamma=z_1^{\gamma_1}\cdots z_m^{\gamma_m},\ |\gamma|=\gamma_1+\cdots+\gamma_m,\ \gamma!=\gamma!\cdots\gamma_m!.$$

For an analytic function $f$ on $\F^m$ (i.e., entire function) given by a formal power series
$$ f=\sum_\gamma a_\gamma z^\gamma $$
with $a_\gamma\in\F$ such that $\lim_{|\gamma|\rightarrow\infty}|a_\gamma|r^{|\gamma|}=0\ (\forall r\in\F^*=\F\setminus\{0\})$, define
$$ |f|_r=\mathrm{sup}_{\gamma} |a_\gamma|r^{|\gamma|}.$$
We denote by $\mathcal E_m$ the ring of all analytic functions on $\F^m$.

We define a meromorphic function $f$ on $\F^m$ to be the quotient of two analytic functions $g,h\in\mathcal E_m$ such that $g$ and $h$ have no common factors in $\mathcal E_m,$ i.e., $f=\frac{g}{h}$. We define
$$ |f|_r=\frac{|g|_r}{|h|_r}.$$
We denote by $\mathcal M_m$ the field of all meromorphic functions on $\F^m$, which is the fractional field of $\mathcal E_m$.

\noindent
{\bf 2.2. Derivatives and Hasse derivatives.} \ 
For a meromorphic function $f\in\mathcal M_m$ and a multi-index $\gamma =(\gamma_1,\ldots,\gamma_m)$, we set
$$ \partial^\gamma f=\frac{\partial^{|\gamma|}f}{\partial z_1^{\gamma_1}\cdots\partial z_m^{\gamma_m}}.$$

Let $\alpha=(\alpha_1,\ldots,\alpha_m)$ and $\beta=(\beta_1,\ldots,\beta_m)$ be multi-indices. We say that $\alpha\ge\beta$ if $\alpha_i\ge\beta_i$ for all $i=1,.\ldots,m.$ If $\alpha\ge\beta$, we define
$$\alpha-\beta=(\alpha_1-\beta_1,\ldots,\alpha_m-\beta_m), \binom{\alpha}{\beta}=\binom{\alpha_1}{\beta_1}\cdots\binom{\alpha_m}{\beta_m}.$$

For an analytic function $f=\sum_{\alpha}a_\alpha z^\alpha$ and a multi-index $\gamma$, we define the Hasse derivative of multi-index $\gamma$ of $f$ by
$$ D^\gamma f=\sum_{\alpha\ge\gamma}\binom{\alpha}{\gamma}a_\alpha z^{\alpha-\gamma}.$$
We may verify that $D^{\alpha}D^{\beta}f=\binom{\alpha+\beta}{\beta}D^{\alpha+\beta}$ for all $f\in\mathcal E_m$. Therefore, the Hasse derivative $D$ can be extended to meromorphic functions in the following way: 
\begin{itemize}
\item For a multi-index $e_i=(0,\ldots,0,\underset{j^{th}-position}{1},0,\ldots,0)$, we set $D^k_jf:=D^{ke_i}(f)$. 
\item For a meromorphic function $f=\frac{g}{h}\ (g,h\in\mathcal E_m)$, we define
$$D^{e_i}=D^1_jf:=\frac{hD^1_ig-gD^1_ih}{h^2}, \ j=1,\ldots,m.$$
\item  For $\gamma=(\gamma_1,\ldots,\gamma_m)$, we may choose a sequence of multi-indices $\gamma=\alpha^1>\alpha^2>\cdots>\alpha^{|\gamma|}$ such that $\alpha^i=\alpha^{i+1}+e_{j_i}\ ({j_i}\in\{1,\ldots,m\})$ for $1\le i\le |\gamma|-1$ and $ \alpha^{|\gamma|}=e_{j_{|\gamma|}}\ (j_{|\gamma|}\in\{1,\ldots,m\})$ and define
$$ D^{\alpha_i}h=\frac{1}{\binom{\alpha_{i+1}+e_{j_i}}{\alpha_{i+1}}}D^{e_{j_i}}D^{\alpha_{i+1}}h, \forall i=|\gamma|-1,|\gamma|-2,\ldots,1. $$
\end{itemize}
We summarize here the fundamental properties of the Hasse derivative from \cite{Yan} as follows:

(i) $D^\gamma (f+g)=D^\gamma f+D^\gamma g,\ f,g\in\mathcal M_m.$

(ii) $D^\gamma (fg)=\sum_{\alpha,\beta}D^\alpha fD^\beta g,\ f,g\in\mathcal M_m.$

(iii) $D^\alpha D^\beta f=\binom{\alpha+\beta}{\beta}D^{\alpha+\beta}f,\ f\in\mathcal M_m$

(iv) (Lemma on the logarithmic derivative) For $f\in\mathcal E_m$,
$$ |D^\gamma f|_r\le\frac{|f|_r}{r^{|\gamma|}},\ |\partial^\gamma f|_r\le\frac{|f|_r}{r^{|\gamma|}}. $$

(v) For $f\in\mathcal E_m$ and a multi-index $\gamma$, let $P$ be an irreducible element of $\mathcal E_m$ that divides $f$ with exact multiplicity $e$. If $e>|\gamma|$, then $P^{e-|\gamma|}$ divides $D^\gamma f$.

For each integer $k\ge 2$, let
$$\mathcal M_m[k] = \{Q\in\mathcal M_m : D^i_jQ\equiv 0\text{ for all }0<i<k\text{ and }1\le j\le m\}.$$
If $F$ has characteristic $0$, then $\mathcal M_m[k]=\F$ for all $k\ge 2$. If $\F$ has characteristic $p>0$ and if $s\ge 1$ is an integer, then $\mathcal M_m[p^s]$ is the fraction field of $\mathcal E_m$, where $\mathcal E_m[p^s]=\{g^{p^s}:g\in\mathcal E_m\}$ is a subring of $\mathcal E_m$. Moreover, 
$$\mathcal M_m[p^{s-1}+1]=\mathcal M_m[p^s].$$

\noindent
{\bf 2.3. Non-Archimedean Nevanlinna's function.} \ 

Let $f=\sum_{\gamma}a_\gamma z^\gamma\in\mathcal E_m$ be an holomorphic function. The counting function of zeros of $f$ is defined as follows:
$$ N_f(0,r)=n_f (0,0)\log r+\int_{0}^r(n_f(0,t)-n_f(0,0))\frac{dt}{t}\ (r>0),$$
where 
$$n_f(0,r)=\mathrm{sup}\{|\gamma|; |a_\gamma|r^{|\gamma|}=|f|_r\}\text{ and }n_f(0,0)=\min\{|\gamma|; a_\gamma\ne 0\}.$$

Let $f$ be a meromorphic function on $\F^m$. Assume that $f=\frac{g}{h}$, where $g,h$ are holomorphic functions without common factors. We define
$$ N_f(0,r)=N_g(0,r)\text{ and } N_f(\infty,r)=N_h(0,r).$$
The Poisson-Jensen-Green formula (see \cite[Theorem 3.1]{CY}) states that
$$ N_f(0,r)-N_f(\infty,r)=\log |f|_r+C_f\ \text{ for all }r>0, $$
where $C_f$ is a constant depending on $f$ but not on $r$.

Suppose that $f\not\equiv a$ for $a\in\F$. The counting function of $f$ with respect to the point $a$ is defined by
$$ N_f(a,r)=N_{f-a}(0,r).$$

The proximity functions of $f$ with respect to $\infty$ and $a$ are defined respectively as follows
$$ m_f(\infty,r)=\max\{0,\log |f|_r\}=\log^+|f|_r\text{ and }m_f(a,r)=m_{1/(f-a)}(\infty,r). $$

The characteristic function of $f$ is defined by
$$ T_f(r)=m_f(\infty,r)+N_f(\infty,r). $$
Note that, if $f=\frac{g}{h}$ as above then $T_f(r)=\max\{\log|g|_r,\log|h|_r\}+O(1)$.

The first main theorem is stated as follows:
$$ T_f(r)=m_f(a,r)+N_f(a,r)+O(1)\ (\forall r>0).$$

\noindent
{\bf 2.4. Truncated counting function.}\ 

 Let $f\in\mathcal E_m$. For $j=1,\ldots,m,$ define
$$ g_j=\mathrm{gcd}(f,D^1_j(f))\text{ and }h_j=\frac{f}{g_j}.$$
The radical $R(f)$ of $f$ is defined to be the least common multiple of $h_j$'s. 

\textit{Case 1: $\F$ has characteristic $p=0$.} The truncated counting function of zeros of $f$ is defined by
$$ N^{(l)}_f(0,r)=N_{\mathrm{gcd}(f,R(f)^{l})}(0,r).$$
In particular,
$$ N^{(1)}_f(0,r)=N_{R(f)}(0,r).$$

\textit{Case 2: $\F$ has characteristic $p>0$.} We define $R_{p^s}(f)$ by induction in $s=0,1,\ldots$ For $s=0$, set $R_{p^0}(f)=R(f)$. For $s\ge 1$, assume that $R_{p^{s-1}}(f)$ has been defined. We set
$$ \overline{f}=\frac{f}{\mathrm{gcd}(f,R_{p^{s-1}}(f)^{p^s})},\ g_i=\mathrm{gcd}(\overline{f},D_i^{p^s}\overline{f}),\ h_i=\frac{\overline{f}}{g_i} $$
for $i=1,\ldots,m$. Let $H$ be the least common multiple of $h_i$'s, and set
$$ G=\frac{H}{\mathrm{gcd}(H,R_{p^{s-1}}(H)^{p^{s-1}})},$$
which is a $p^s$th power. Let $R$ be the $p^s$th root of $G$ and define the higher $p^s$-radical $R_{p^s}(f)$ of $f$ to be the least common multiple of $R_{p^{s-1}}(f)$ and $R$. 

Take a sequence $\{r_j\}_{i\in\mathbb N}\subset|\F^*|$ such that $r_j\rightarrow\infty.$ Take $s_j$ such that if $P\in\mathcal E_m$ is irreducible such that $P|f$ and $P$ is not unit on $\B^m(r_j)$ then $P|R_{p^s}(f)$ for $s>s_j$. Let $u_{j}$ be a unit on $\B^m(r_j)$ such that
$$ R_{p^{s_j}}(f)=u_jR_{p^{s_{j+1}}}(f).$$
Define $v_j=\prod_{l=j}^\infty u_j$, which is unit on $\B^m(r_j)$, and
$$ S(f)=\lim_{j\rightarrow\infty}\frac{R_{p^{s_j}}(f)}{v_j}\in\mathcal E_m,$$
which is called the square free part of $f$. The truncated (to level $l$) counting function of zeros of $f$ is defined by
$$ N^{(l)}_f(0,r)=N_{\mathrm{gcd}(f,S(f)^l)}(0,r).$$

\noindent
{\bf 2.5. Non-Archimedean meromorphic maps and family of hypersurfaces.}\ 

Let $V$ be a projective subvariety of $\P^M(\mathbb F)$ of dimension $n\ (n\le M)$. For a positive integer $d$, take a basis $\{[A_1],\ldots,[A_{H_d(V)}]\}$ of $I_d(V)$, where $A_i\in\mathcal H_d[x_0,\ldots,x_M]$. Let $f:\F^m\rightarrow\P^M(\F)$ be a non-Archimedean meromorphic map with a reduced representation $\bbf=(f_0,\ldots,f_M)$, which is non-degenerate over $I_d(V)$. We have the following definition.

\begin{definition}\label{def3}
Assume that $\F$ has the character $p>0$. Denote by $s$ the smallest integer such that any subset of $\{A_1(\bbf),\ldots,A_{H_d(V)}(\bbf)\}$ linearly independent over $\F$ remains linearly independent over $\mathcal M_m[p^{s}]$. We call $s$ is the $d^{th}$-index of non-degeneracy of $f$.
\end{definition}
We see that the above definition does not depend on the choice of the basis $\{[A_i];1\le i\le H_d(V)\}$ and the choice of the reduced representation $\bbf$. If $V=\P^M(\F)$ and $d=1$ then $s$ is also called the index of independence of $f$ (see \cite[Definition 4.1]{Yan}).

The following three lemmas are proved in \cite{AQ} for the case of $\F=\C$ and the canonical absolute value. However, with the same proof, they also hold for arbitrary algebraic closed field $\F$ of character $p\ge 0$ and complete with an arbitrary absolute value. We state them here without the proofs.

Throughout this paper, we sometimes identify each hypersurface in a projective variety with its defining homogeneous polynomial. The following lemma of An-Quang \cite{AQ} may be considered as a generalization of the lemma on Nochka weights in \cite{Noc83}.

 \begin{lemma}[{cf. \cite[Lemma 3]{AQ}}]\label{lem4}
Let $V$ be a projective subvariety of $\P^M(\mathbb F)$ of dimension $n\ (n\le M)$. Let $Q_1,\ldots,Q_q$ be $q\ (q>2N-k+1)$ hypersurfaces in $\P^M(\mathbb F)$ in $N$-subgeneral position with respect to $V$ of the common degree $d.$ Then there are positive rational constants $\omega_i\ (1\le i\le q)$ satisfying the following:

i) $0<\omega_i \le 1,\  \forall i\in\{1,\ldots,q\}$,

ii) Setting $\tilde \omega =\max_{j\in Q}\omega_j$, one gets
$$\sum_{j=1}^{q}\omega_j=\tilde \omega (q-2N+n-1)+n+1.$$

iii) $\dfrac{n+1}{2N-n+1}\le \tilde\omega\le\dfrac{n}{N}.$

iv) For $R\subset \{1,\ldots,q\}$ with $\sharp R = N+1$, then $\sum_{i\in R}\omega_i\le n+1$.

v) Let $E_i\ge 1\ (1\le i \le q)$ be arbitrarily given numbers. For $R\subset \{1,\ldots,q\}$ with $\sharp R = N+1$,  there is a subset $R^o\subset R$ such that $\sharp R^o=\rank_{\F}\{[Q_i]; i\in R^o\}=n+1$ and 
$$\prod_{i\in R}E_i^{\omega_i}\le\prod_{i\in R^o}E_i.$$
\end{lemma}

Let $Q$ be a hypersurface in $\P^n(\F)$ of degree $d$ defined by $\sum_{I\in\mathcal I_d}a_{I}x^I=0,$
where $\mathcal I_d=\{(i_0,\ldots,i_M)\in \mathbb N_0^{M+1}\ :\ i_0+\cdots + i_M=d\}$, $I=(i_0,\ldots,i_M)\in\mathcal I_d,$ $x^I=x_0^{i_0}\cdots x_M^{i_M}$ and $(x_0:\cdots: x_M)$ is homogeneous coordinates of $\P^M(\F)$.
Let $f$ be an non-Archimedean meromorphic map from $\F^m$ into a projective subvariety $V$ of $\P^M(\F)$ with a reduced representation $\bbf=(f_0,\ldots ,f_M)$. We define
$$ Q(\bbf)=\sum_{I\in\mathcal I_d}a_{I}f^I ,$$
where $f^I=f_0^{i_0}\cdots f_n^{i_n}$ for $I=(i_0,\ldots,i_n)$. We have the following lemma.

\begin{lemma}[{cf. \cite[Lemma 4]{AQ}}]\label{lem5}
Let $\{Q_i\}_{i\in R}$ be a set of hypersurfaces in $\P^n(\F)$ of the common degree $d$ and let $f$ be a meromorphic mapping of $\F^m$ into $\P^n(\F)$ with a reduced representation $\bbf=(f_0,\ldots ,f_M)$. Assume that $\bigcap_{i\in R}Q_i\cap V=\varnothing$. Then,  there exist positive constants $\alpha$ and $\beta$ such that
$$\alpha \|\bbf\|_r^d \le  \max_{i\in R}|Q_i(\bbf)|_r\le \beta \|\bbf\|_r^d \text{ for any }r>0.$$
\end{lemma} 

\begin{lemma}[{cf. \cite[Lemma 5]{AQ}}]\label{lem6}
Let $\{Q_i\}_{i=1}^q$ be a set of $q$ hypersurfaces in $\P^M(\F)$ of the common degree $d$. Then there exist $(H_V(d)-n-1)$ hypersurfaces $\{T_i\}_{i=1}^{H_V(d)-n-1}$ in $\P^M(\F)$ such that for any subset $R\in\{1,\ldots,q\}$ with $\sharp R=\rank_{\F} \{[Q_i]; i\in R\}=n+1,$ we get $\rank_{\F}\{\{[Q_i]; i\in R\}\cup\{[T_i]; 1\le i\le H_d(V)-n-1\}\}=H_V(d).$
\end{lemma}

\noindent
{\bf 2.5. Value distribution theory for non-Archimedean meromorphic maps.}\ 

Let $f : \F^m \rightarrow V\subset\P^M(\F)$ be a non-Archimedean meromorphic map with a reduced representation $\bbf = (f_0,\dots,f_N)$.
The characteristic function of $f$ is defined by 
\begin{align*}
T_f(r)= \log\|\bbf\|_r,
\end{align*}
where $\|\bbf\|_r=\max_{1\le 0\le n}|f_i|_r$. This definition is well-defined upto a constant.

Let $Q$ be a hypersurface in $\P^n(\F)$ of degree $d$ defined by $\sum_{I\in\mathcal I_d}a_{I}x^I=0,$
where $a_I\in\F \ (I\in\mathcal I_d)$ and are not all zeros. If $Q(\bbf)\not\equiv 0$ then we define the proximity function of $f$ with respect to $Q$ by
$$ m_f(Q,r)=\log\frac{\|\bbf\|^d_r\cdot\|Q\|}{|Q(\bbf)|_r},$$
where $\|Q\|:=\max_{I\in\mathcal I_d}|a_I|$. We see that the definition of $m_f(Q,r)$ does not depend on the choices of the presentations of $f$ and $Q$.

The truncated (to level $l$) counting function of $f$ with respect to $Q$ is defined by
$$ N_f^{(l)}(Q,r):=N^{(l)}_{Q(\bbf)}(0,r).$$
For simplicity, we will omit the character $^{(l)}$ if $l=\infty$.

The first main theorem for non-Archimedean meromorphic maps states that
$$ dT_f(r)=m_f(Q,r)+N_f(Q,r)+O(1).$$

\begin{proposition}[{cf. \cite[Propositions 4.3, 4.4]{Yan}}]\label{pos7}
Let $p$ be the character of $\F$. Assume that $f: \F_m\rightarrow\P^n(\F)$ is a non-Achimedean meromorphic map, which is linearly non-degenerate over $\F$, with a reduced representation $\bbf=(f_0,\ldots,f_n)$. Then there exist multi-indices $\gamma^0=(0,\ldots,0),\gamma^1,\ldots,\gamma^n$ with
$$ |\gamma^0|\le\cdots\le|\gamma^n|\le\kappa_0\le\begin{cases}
p^{s-1}(n-k+1)&\text{ if }p>0,\\
n-k+1&\text{ if }p=0
\end{cases} $$
where $s$ is the index of independence of $f$ and $k=\rank f$, such that the generalized Wronskian
$$ W_{\gamma^0,\ldots,\gamma^n}(f_0,\ldots,f_n)=\det\left(D^{\gamma^i}f_j\right)_{0\le i,j\le n}\not\equiv 0.$$
\end{proposition}
Here $\rank f$ is defined by
$$\rank f=\rank_{\mathcal M_m}\{(D^{\gamma}f_0,\ldots,D^{\gamma}f_n); |\gamma|\le 1\}-1.$$

\section{Proof of main theorems}

\begin{proof}[Proof of Theorem \ref{1.1}]
By replacing $Q_i$ with $Q_i^{d/d_i}$ if necessary, we may assume that all $Q_i\ (i=1,\ldots,q)$ do have the same degree $d$.
It is easy to see that there is a positive constant $\beta$ such that $\beta \|\bbf\|^d\ge |Q_i(\bbf)|$ for every $1\le i\le q.$
Set $ Q:=\{1,\cdots ,q\}$. Let $\{\omega_i\}_{i=1}^q$ be as in Lemma \ref{lem4} for the family $\{Q_i\}_{i=1}^q$.  Let $\{T_i\}_{i=1}^{H_d(V)-n-1}$ be $(H_d(V)-n-1)$ hypersurfaces in $\P^M(\F)$, which satisfy Lemma \ref{lem6}. 

Take a $\F$-basis $\{[A_i]\}_{i=1}^{H_V(d)}$ of $I_d(V)$, where $A_i\in H_d$. Since $f$ is non-degenerate over $I_d(V)$, it implies that $\{A_i(\bbf); 1\le i\le H_V(d)\}$ is linearly independent over $\F$. By Proposition \ref{pos7}, there multi-indices $\{\gamma^1=(0,\ldots,0),\gamma^2\cdots ,\gamma^{H_V(d)}\}\subset  \mathbb Z_+^m$ such that $ |\gamma^0|\le\cdots\le|\gamma^{H_d(V)}|\le\kappa_0,$ where
$$\kappa_0\le\begin{cases}
p^{s-1}(H_V(d)-k)&\text{ if }p>0,\\
H_d(V)-k&\text{ if }p=0
\end{cases}$$
and the generalized Wronskian
$$ W=\det\left(D^{\gamma^i}A_j(\bbf)\right)_{1\le i,j\le H_d(V)}\not\equiv 0.$$
Here, we note that 
\begin{align*}
k&=\rank_{\mathcal M_m}\{(D^{\gamma}f_0,\ldots,D^{\gamma}f_M); |\gamma|\le 1\}-1\\
&=\rank_{\mathcal M_m}\left\{\left(D^{\gamma}\bigl(\frac{f_1}{f_0}\bigl),\ldots,D^{\gamma}\bigl(\frac{f_M}{f_0}\bigl)\right); |\gamma|\le 1\right\}\\
&\le\rank_{\mathcal M_m}\left\{\left(D^{\gamma}\bigl(\frac{A_2(\bbf)}{A_1(\bbf)}\bigl),\ldots,D^{\gamma}\bigl(\frac{A_{H_d(V)}(\bbf)}{A_1(\bbf)}\bigl)\right); |\gamma|\le 1\right\}\\
&=\rank_{\mathcal M_m}\left\{(D^\gamma(A_1(\bbf)),\ldots,D^\gamma(A_{H_d(V)}(\bbf))); |\gamma|\le 1\right\}-1.
\end{align*}

For each $R^o=\{r^0_1,\ldots,r^0_{n+1}\}\subset\{1,\ldots,q\}$ with $\rank_{\F} \{Q_i\}_{i\in R^o}=\sharp R^o=n+1$, set 
$$W_{R^o}\equiv\det\bigl (D^{\gamma^j}Q_{r^0_v}(\bbf) (1\le v\le n+1),D^{\gamma^j}T_l(\bbf) (1\le l\le H_V(d)-n-1)\bigl )_{1\le j\le H_V(d)}.$$
Since $\rank_{\F}\{[Q_{r^0_v}] (1\le v\le n+1),[T_l] (1\le l\le H_V(d)-n-1)\}=H_V(d)$, there exists a nonzero constant  $C_{R^o}\in\F$ such that $W_{R^o}=C_{R^o}\cdot W$. 

We denote by $\mathcal R^o$ the family of all subsets $R^o$ of $\{1,\ldots,q\}$ satisfying 
$$\rank_{\F}\{[Q_i];i\in R^o\}=\sharp R^o=n+1.$$

For each $r>0$, there exists $\bar R\subset Q$ with $\sharp \bar R=N+1$ such that $|Q_{i}(\bbf)|_r\le |Q_j(\bbf)|_r,\forall i\in \bar R,j\not\in \bar R$. We choose $R^{o}\subset R$ such that $R^o\in\mathcal R^o$ and $R^o$ satisfies Lemma \ref{lem4}(v) with respect to numbers $\bigl \{\dfrac{\beta \|\bbf\|_r^d}{|Q_i(\bbf)|_r}\bigl \}_{i=1}^q$. Since $\bigcap_{i\in \bar R}Q_i=\varnothing$, by Lemma \ref{lem5}, there exists a positive constant $\alpha^{\bar R}$ such that
$$ \alpha^{\bar R} \|\bbf\|_r^d\le \max_{i\in \bar R}|Q_i(\bbf)|_r. $$
Then, we get
\begin{align*}
\dfrac{\|\bbf\|_r^{d(\sum_{i=1}^q\omega_i)}|W|_r}{|Q_1(\bbf)|_r^{\omega_1}\cdots |Q_q(\bbf)|_r^{\omega_q}}
&\le\dfrac{|W|_r}{\alpha^{q-N-1}_{\bar R}\beta^{N+1}}\prod_{i\in \bar R}\left (\dfrac{\beta\|\bbf\|_r^d}{|Q_i(\bbf)|_r}\right )^{\omega_i}\\
&\le A_{\bar R}\dfrac{|W|_r\cdot \|\bbf\|_r^{d(n+1)}}{\prod_{i\in \bar R^o}|Q_i(\bbf)|_r}\\
&\le B_{\bar R}\dfrac{|W_{\bar R^o}|_r\cdot \|\bbf\|_r^{dH_V(d)}}{\prod_{i\in \bar R^o}|Q_i(\bbf)|_r\prod_{i=1}^{H_V(d)-n-1}|T_i(\bbf)|_r},
\end{align*}
where $A_{\bar R}, B_{\bar R}$ are positive constants. 

Therefore, for every $r>0$,
\begin{align*}
\log\dfrac{\|\bbf\|_r^{d(\sum_{i=1}^q\omega_i-H_d(V)}|W|_r}{|Q_1(\bbf)|_r^{\omega_1}\cdots |Q_q(\bbf)|_r^{\omega_q}}&\le\max_{R}\log\dfrac{|W_{R}|_r}{\prod_{i\in R}|Q_i(\bbf)|_r\prod_{i=1}^{H_V(d)-n-1}|T_i(\bbf)|_r}+O(1)\\
&\le -\sum_{j=1}^{H_d(V)}|\gamma^j|\log r +O(1),
\end{align*}
where the maximum is taken over all subsets $R\subset\{1,\ldots,q\}$ such that $\sharp R=n+1$ and $\rank_{\F}\{[Q_i];i\in R\}=n+1$. Here, the last inequality comes from the lemma on logarithmic derivative. By the Poisson-Jensen-Green formula, the definitions of the approximation function and the characteristic function, we have
$$ \sum_{i=1}^q\omega_im_f(Q_i,r)-dH_d(V)T_f(r)-N_W(0,r)\le -(H_d(V)-1)\log r +O(1),$$
(note that $\sum_{i=1}^{H_d(V)}|\gamma^i|\le H_d(V)-1$). Then, by the first main theorem, we obtain
\begin{align}\label{new2}
 (\sum_{i=1}^q\omega_i-H_d(V))dT_f(r)\le\sum_{i=1}^q\omega_iN_{f}(Q_i,r)-N_W(0,r)-(H_d(V)-1)\log r+O(1).
\end{align}

\textbf{Claim.} $\sum_{i=1}^q\omega_iN_{f}(Q_i,r)-N_{W}(0,r)\le \sum_{i=1}^q\omega_iN^{(\kappa_0)}_{f}(Q_i,r)+O(1)$.

Indeed, set $\tilde G_j=\mathrm{gcd}(Q_j(\bbf),S(Q_j(\bbf))^{\kappa_0})$. Since $\omega_i\ (1\le i\le q)$ are rational numbers, there exists an integer $A$ such that $\tilde\omega_i=A\omega_i\ (1\le i\le q)$ are integers. 

Let $P\in\mathcal E_m$ be an irreducible element with $P|\prod_{i=1}^qQ_i(\bbf)^{\tilde\omega_i}$. There exists a subset $R$ of $\{1,\ldots,q\}$ with $\sharp R=N+1$ such that $P$ is not a division of $Q_i(\bbf)$ for any $i\not\in R$. Denote by $e_i$ the largest integer such that $P^{e_i}|Q_i(\bbf)$ for each $i\in R$. Then, there is a subset $R^o\subset R$ with $\sharp R^o=n+1$, $W_{R^o}\not\equiv 0$ and
$$\sum_{i\in R}\omega_i\max\{0,e_i-\kappa_0\}\le\sum_{i\in R^o}\max\{0,e_i-\kappa_0\}.$$
Also, since $W=C_{R^o}\cdot W_{R^o}$, it clear that $P$ divides $W$ with multiplicity at least
\begin{align*}
\min_{\{j_1,\ldots,j_{n+1}\}\subset\{1,\ldots,H_d(V)\}}\sum_{i\in R^0}\min\{0,e_i-|\gamma^{j_i}|\}&\ge\sum_{i\in R^0}\min\{0,e_i-\kappa_0\}\\
&\ge \sum_{i\in R}\omega_i\max\{0,e_i-\kappa_0\}\\
&=\sum_{i\in R}\omega_i(e_i-\min\{e_i,\kappa_0\}).
\end{align*}
This implies that
$$ P^{\sum_{i\in R}\tilde\omega_i e_i}|W^{A}\cdot P^{\sum_{i\in R}\tilde\omega_i\min\{e_i,\kappa_0\}}.$$
We note that $P^{\tilde\omega_i\min\{e_i,\kappa_0\}}|G_i^{\tilde\omega_i}$. Therefore,
$$ P^{\sum_{i\in R}\tilde\omega_i e_i}|W^{A}\cdot \prod_{i\in R}G_i^{\tilde\omega_i}.$$
This holds for every such irreducible element $P$. Then it yields that
$$ \prod_{i=1}^qQ_i(\bbf)^{\tilde\omega_i}|W^A\cdot\prod_{i=1}^qG_i^{\tilde\omega_i}.$$
Hence,
$$ \sum_{i=1}^qN_{f}(Q_i,r)\le N_W(0,r)+\sum_{i=1}^qN^{(\kappa_0)}_{f}(Q_i,r).$$
The claim is proved.

From the claim, Lemma \ref{lem4}(ii) and the inequality (\ref{new2}), we obtain
\begin{align*}
(\tilde\omega(q-2N+n-1)&-H_d(V)+n+1)dT_f(r)\\
&\le\sum_{i=1}^q\omega_iN^{(\kappa_0)}_{f}(Q_i,r)-(H_d(V)-1)\log r+O(1).
\end{align*}
Note that, $\omega_i\le\tilde\omega (1\le i\le q)$ and $\dfrac{n+1}{2N-n+1}\le\tilde\omega\le\dfrac{n}{N}$. Then, the above inequality implies that
$$\left (q-\frac{(2N-n+1)H_d(V)}{n+1}\right)\le\sum_{i=1}^q\dfrac{1}{d}N^{(\kappa_0)}_{f}(Q_i,r)-\dfrac{N(H_d(V)-1)}{nd}\log r+O(1).$$
The theorem is proved.
\end{proof}

\begin{proof}[Proof of Theorem \ref{1.2}]
For $r>0$, without loss of generality, we may assume that 
$$ |Q_1(\bbf)|_r^{1/\deg Q_1}\le|Q_2(\bbf)|_r^{1/\deg Q_2}\le\cdots\le |Q_q(\bbf)|_r^{1/\deg Q_{N+1}}.$$
Since $\bigcap_{i=1}^{N+1}Q_i=\varnothing$, by Lemma \ref{lem5}, there exists a positive constant $C$ such that
$$ C \|\bbf\|_r\le \max_{1\le i\le N+1}|Q_i(\bbf)|_r^{1/\deg Q_i}=|Q_{N+1}(\bbf)|_r^{1/\deg Q_{N+1}}. $$
Then, we get
\begin{align*}
\sum_{i=1}^q\dfrac{m_f(Q_i,r)}{\deg Q_i}&=\log\dfrac{\|\bbf\|_r^{q}}{|Q_1(\bbf)|_r^{1/\deg Q_1}\cdots |Q_q(\bbf)|_r^{1/\deg Q_q}}+O(1)\\
&\le\log\prod_{i=1}^N\dfrac{\|\bbf\|_r}{|Q_i(\bbf)|_r^{1/\deg Q_i}}+O(1)\\
&=\sum_{i=1}^N\dfrac{m_f(Q_i,r)}{\deg Q_i}+O(1)\\
&\le N\cdot T_f(r)+O(1).
\end{align*}
Therefore,
$$ (q-N)T_f(r)\le\sum_{i=1}^q\dfrac{1}{\deg Q_i}N_{f}(Q_i,r)+O(1)\ \ (r>0).$$
The theorem is proved.
\end{proof}

\end{document}